\documentclass[a4paper,11pt]{amsart}
\usepackage{etex}
\usepackage[top=30 mm,bottom=30 mm,right=25 mm,left=25 mm]{geometry}
\makeatletter
\@namedef{subjclassname@2020}{\textup{2020} Mathematics Subject Classification}
\makeatother
\usepackage{graphicx, mathrsfs}
\usepackage{tikz-cd}
\usepackage[all,poly]{xy}
\usepackage{amscd,amsmath,amsthm,amsfonts,latexsym,amssymb, mathrsfs,stmaryrd,graphics,pst-all,amscd,yfonts,mathtools,graphicx,wrapfig}
\usepackage{amsfonts}
\usepackage{amssymb}
\usepackage{amsmath}
\usepackage{color}
\usepackage[colorlinks]{hyperref}
\usepackage{enumerate}
\scrollmode    
\theoremstyle{plain}
\newtheorem{theorem}{Theorem}[section]
\newtheorem{lemma}[theorem]{Lemma}
\newtheorem{corollary}[theorem]{Corollary}
\newtheorem{proposition}[theorem]{Proposition}

\theoremstyle{definition}
\newtheorem{definition}[theorem]{Definition}

\newtheorem{example}[theorem]{Example}

\theoremstyle{remark}

\newtheorem{remarks}[theorem]{Remarks}

\numberwithin{equation}{section}

\newcommand{\forget}[1]{}

\begin{document}
	\emergencystretch 3em
	\title[On J-torsionless modules]{On J-torsionless modules}
	
	\author{Dimpy Mala Dutta, A. M. Buhphang* and M. B. Rege }

	\address{Dimpy M. Dutta, Department of Mathematics, North-Eastern Hill University, Shillong-793022, Meghalaya, India.}
	\email{dimpymdutta@gmail.com}

	\address{A. M. Buhphang, Department of Mathematics, North-Eastern Hill University, Shillong-793022, Meghalaya, India.}
	
	\email{ardeline17@gmail.com}
	
	\address{M. B. Rege, Retired professor, Department of Mathematics, North-Eastern Hill University,
		 Shillong-793022, Meghalaya, India.}
	\email{mb29rege@gmail.com}

	\begin{abstract}
		In this paper, we introduce the concept of JReject of a class of modules as a generalization of the notion of reject of a class of modules. We also introduce the notion of J-torsionless modules and give a characterization of regularity on the basis of the J-torsionless condition. A necessary and sufficient condition on $R$ is also given for every cyclic module over $R$ to be J-torsionless. Finally, we give a description of self-injective rings over which every module is J-torsionless.
	\end{abstract}
	\maketitle
	\section{Introduction}
	In this paper, unless otherwise stated, $R$ will denote a unital and associative ring. $M^*$ will denote the dual right $R$-module $Hom_R(M,R)$. {\em Nil(R)}, $I(R)$ and {\em J(R)} will denote the set of nilpotent elements, the set of idempotent elements and the Jacobson radical of $R$, respectively. By an $R$-module we mean a left $R$-module and we denote the (left) $R$-modules by $_RM$. 	By $r_R(a)$ (respectively, $l_R(a)$) we mean the set of right (respectively, left) annihilators of $a$ in $R$. Moreover, all modules and homomorphisms are assumed to be left $R$-modules and left $R$-homomorphisms, respectively. For consistency of notation, we shall denote homomorphisms on the side opposite to that of scalars. Anderson and Fuller in \cite{Anderson} defined the {\em reject} of  a class of modules $\mathscr{U}$  in a module $M$, $Rej_M(\mathscr{U})$ as follows.
	\[Rej_M(\mathscr{U})= \cap \lbrace Kerf \ \vert 
	\ f: M \rightarrow U \ \text{for some} \ U \in \mathscr{U} \rbrace.\]
	It was also shown in \cite{Anderson} that if we consider the class of simple left $R$-modules and denote it by $\mathscr{S}$,  then for each module $_RM$, 
	\[Rad(M)=Rej_M(\mathscr{S}),\]
	where Rad(M) is the radical of M. 
	We introduce the concept of \textit{JReject} of a class of modules $\mathscr{U}$ in a module $M$ denoted by $JRej_M(\mathscr{U})$ as follows:
	\[JRej_M(\mathscr{U}) = \lbrace m\in M \ \vert \ mf \in Rad(U), \ \forall \ f: M \rightarrow U \ \text{for some} \ U \in \mathscr{U}\rbrace. \]
	Clearly, if $\mathscr U$ is a singleton set $\{U\}$, we have 
	$$JRej_M(U) = \lbrace m\in M \ \vert \ mf \in Rad(U), \ \forall \ f: M \rightarrow U\rbrace.$$
	Notice that $JRej_M(\mathscr U)$ is a submodule of $M$.
	\section{Results}
	
	It can be easily observed that $Rej_M(\mathscr{U}) \subseteq JRej_M(\mathscr{U})$. However, in general the converse may not hold for arbitrary classes of modules. For instance, consider $\mathscr U$ to be the singleton set $\lbrace \mathbb{Z}_4 \rbrace$ with the underlying ring $R$ to be $\mathbb Z_4$  and $M=\mathbb{Z}_4 \times \mathbb{Z}_4$. Observe that $(\bar{2}, \bar0) \in JRej_M(\mathbb Z_4)$, whereas $(\bar2, \bar0) \notin Rej_M(\mathbb Z_4)$ as $(\bar2, \bar0) \notin Ker (\pi_1),$ where $\pi_1$ is the projection map on the $1^{st}$ component. 
	
	The natural question that arises from the above discussion is that when will these two submodules of $M$ coincide? To provide some insight into this question we consider the following remarks.
	
	\begin{remarks}\label{rejrem1}
		\begin{enumerate}
			\item $Rej_M(\mathscr S) = JRej_M(\mathscr S)$. 
			\item For a semiprimitive ring $R$ and an $R$-module $M$, $$Rej_M(R)=JRej_M(R).$$
		\end{enumerate}
	\end{remarks}
	Recall from \cite{Anderson} that a module $M$ is (finitely) cogenerated by a class of modules $\mathscr U$  if there exists a (finite) indexed set $(U_\alpha)_{\alpha \in A}$ in $\mathscr U$ and a monomorphism 
	$$0 \rightarrow M \xrightarrow{f} \prod\limits_A U_\alpha.$$
	
	\begin{theorem}
		For a class of modules $\mathscr U$, there exists a subclass $\mathscr U^\prime$ of $\mathscr U$ such that for a module $M$, $JRej_M(\mathscr U)$ is the smallest submodule $L$ of $M$ such that $M/L$ is cogenerated by $\mathscr U^\prime$
	\end{theorem}
	
	\begin{proof}
		We define the subclass $\mathscr U^\prime$ of $\mathscr U$ as the collection of modules of the form $U/Rad(U)$, where $U \in \mathscr U$. It is clear that there is  an indexed set $(U_\alpha / Rad(U_\alpha))_{\alpha \in A}$ in $\mathscr U^\prime$ and homomorphisms $g_\alpha : M \rightarrow U_\alpha / Rad(U_\alpha)$ for each $\alpha \in A$ with $\bigcap_A Ker(g_\alpha)=JRej(\mathscr U)$. Hence, the natural homomorphism $\prod_Ag_\alpha : M \rightarrow \prod_A U_\alpha / Rad(U_\alpha)$ has kernel $JRej_M(\mathscr U)$  \cite[Corollary 6.2]{Anderson}, implying that $JRej_M(\mathscr U)$ is cogenerated by $\mathscr U^\prime$.
		
		Next let $L$ be a module such that $M/L$ is cogenerated by $\mathscr U^\prime$.  Assume that $(U_\alpha / Rad(U_\alpha))_{\alpha \in A}$ is an indexed set in $\mathscr U^\prime$ and $g:M \rightarrow \prod_A U_\alpha/ Rad(U_\alpha)$ be a homomorphism with $L=Ker(g)$, then $L=\bigcap_A Ker(g \pi_\alpha) \supseteq JRej_M(\mathscr U)$. Thus if $M/L$ is cogenerated by $\mathscr U^\prime$, $L \supseteq JRej_M(\mathscr U).$
	\end{proof}
	
	For a module $M$, a class of modules $\mathscr U$ and its subclass $\mathscr U^\prime$ of  modules $U/Rad(U)$, $U \in \mathscr U$, we have the following corollary, which is an analogue to Corollary 8.14 of \cite{Anderson}.
	\begin{corollary}\label{corgen}
		$\mathscr U^\prime$ cogenerates $M$ iff $JRej_M(\mathscr U)=0$.
	\end{corollary}

	We have already discussed the relation between the reject and the JReject of a given class of modules in a particular module M. This prompts us to question as to how the reject and the JReject of two distinct classes of modules are connected while the underlying module remains the same? Are there specific conditions under which we can expect to obtain a comparable outcome?  While there may be many answers to this question, (for example \cite[Lemma 8.19(2)]{Anderson}), in the following proposition we record another response to it in terms of a subclass of one of the two given classes of modules.
	
	Let $\mathscr U$ and $\mathscr V$ be two classes of modules and $\mathscr V^\prime$ be the subclass of $\mathscr V$ consisting of modules $V/Rad(V)$ for $V \in \mathscr V$. Then, 
	\begin{proposition}
			If $\mathscr V^\prime \subseteq Cog(\mathscr U)$ then, $ Rej_M(\mathscr U) \leq JRej_M(\mathscr V).$
	\end{proposition}
	\begin{proof}
		Take $ m \in Rej_M(\mathscr U)$ such that  $m \notin JRej_M(\mathscr V)$. Thus there is $h \in Hom(M,V)$ for some $V \in \mathscr V$ with $mh \notin Rad(V)$. Given $\mathscr V^\prime \subseteq Cog(\mathscr U)$, there exists a homomorphism $f:V/Rad(V) \rightarrow U$ for some $U\in \mathscr U$ such that $mh+Rad(V) \notin Kerf$. Considering the natural projection map $\eta: V \rightarrow V/Rad(V)$, we have $h \eta f \in Hom(M,U)$ with $m \notin Ker (h \eta f)$. This gives us a contradiction that $m \notin Rej_M(\mathscr U)$. Therefore, $m \in JRej_M(\mathscr V)$.
	\end{proof}
		We now obtain a result which relates for a given module $M$, the JReject of a class $\mathscr U$ of modules in $M$ and the JReject of $\mathscr U$ in a factor module of $M$. 
		\begin{proposition}
			Let $M$ be any module and $\mathscr U$ be a class of modules. If $L \leq M$, such that $L \leq JRej_M(\mathscr U)$ and $JRej_{M/L}(\mathscr U)$=0, then $L=JRej_M(\mathscr U)$.
		\end{proposition}
		\begin{proof}
			Given $L \leq JRej_M(\mathscr U),$ we need to prove that $JRej_M(\mathscr U) \leq L$. If possible, choose $m \in JRej_M(\mathscr U)$ such that $m \notin L$. This gives us the implication $m+L \nsubseteq JRej_{M/L}(\mathscr U)$, further implying that there exists $h \in Hom(M/L,U)$ for some $U \in \mathscr U$ with $(m+L)h \notin Rad(U)$. If $\eta$ is the natural projection map from $M$ to $M/L$, it is observed that $\tilde{h}= \eta h \in Hom(M,U)$. Again, $m \tilde{h} \notin Rad(U)$, suggesting $m \notin JRej_M(\mathscr U)$, which contradicts our assumption. Hence $m\in L$, yielding that $L=JRej_M(\mathscr U)$.
		\end{proof}
	In what follows, we show that the structure of JReject of a class of modules is preserved under a module homomorphism.
	\begin{proposition}\label{projr}
		Let $\mathscr U$ be a class of modules and $M$, $N$ be any two modules. Then for a homomorphism $h:M \rightarrow N$, $$JRej_M(\mathscr U)h \leq JRej_N(\mathscr U)).$$
	\end{proposition}
	\begin{proof}
		First we observe that for $g \in Hom(N,U)$, $hg \in Hom(M,U)$ for any $U \in \mathscr U$. Let $m \in JRej_M(\mathscr U)$. Then for $U \in \mathscr U$ and each $g \in Hom(N,U)$, $mhg \in Rad(U)$. Thus $JRej_M(\mathscr U)h \leq JRej_N(\mathscr U)$.
	\end{proof}
	\begin{corollary}
		If $h:M \rightarrow N$ is an epimorphism such that $Ker( h) \subseteq Rej_M(\mathscr U)$, then $JRej_M(\mathscr U)h = JRej_N(\mathscr U).$
\end{corollary}
	\begin{proof}
			Let $n \in JRej_N(\mathscr U)$. Since $h$ is epic, there exists $m \in M$ such that $mh=n$. Now for an arbitrary homomorphism $f: M \rightarrow U$, $Rej_M(\mathscr U) \subseteq Ker (f)$ and so by hypothesis, $Ker(h) \subseteq Rej_M(\mathscr U) \subseteq Ker (f) $. By the Factor Theorem \cite[Theorem 3]{Anderson}, there exists a unique $g \in Hom(N,U)$ such that $f=hg$. Now, $mf=mhg=0$. Since for any $U \in \mathscr U$, the homomorphism $f$ was arbitrary, we may conclude that $m \in Rej_M(\mathscr U)$, implying $JRej_N(\mathscr U) \leq JRej_M(\mathscr U)h$. Thus using Proposition \ref{projr}, $JRej_M(\mathscr U)h = JRej_N(\mathscr U)$.
		\end{proof}
	
The next proposition proves that JReject of a class of modules $\mathscr U$ distributes  over direct sums of modules.
	\begin{proposition}
		If $(N_\alpha)_{\alpha \in A}$ is an indexed set of modules, then
		$$JRej_{\left(\bigoplus_A N_\alpha\right)}(\mathscr U)=\bigoplus_AJRej_{N_\alpha}(\mathscr U).$$
	\end{proposition}
	\begin{proof}By Proposition \ref{projr} we have the following two relations $$JRej_{\left(\bigoplus_A N_\alpha\right)}(\mathscr U) = \sum  JRej_{\left(\bigoplus_A N_\alpha\right)}(\mathscr U)  \pi_\alpha i_\alpha \leq \sum  JRej_{N_\alpha}(\mathscr U) i_\alpha,$$
			$$\text{and,}  \hspace{.5cm} JRej_{N_\alpha}(\mathscr U) i_\alpha \leq JRej_{\left(\bigoplus_A N_\alpha\right)}(\mathscr U),$$
		where  for each $\alpha \in A$, $i_\alpha$ and $\pi_\alpha$ are the natural injection and projection maps respectively. The second equation further implies that
		$$\sum JRej_{N_\alpha}(\mathscr U) i_\alpha  \leq JRej_{\left(\bigoplus_A N_\alpha\right)}(\mathscr U),$$
		i.e., 
		$$\sum JRej_{N_\alpha}(\mathscr U) i_\alpha = JRej_{\left(\bigoplus_A N_\alpha\right)}(\mathscr U).$$
		But,  $$\sum JRej_{N_\alpha}(\mathscr U) i_\alpha = \bigoplus_AJRej_{N_\alpha}(\mathscr U).$$
		Consequently,  $JRej_{\left(\bigoplus_A N_\alpha\right)}(\mathscr U)=\bigoplus_AJRej_{N_\alpha}(\mathscr U).$
	\end{proof}
	We shall use the notation $l_R(m)~ (\text{respectively,}~ r_R(m))$ for the set of all left (respectively, right) annihilators of an element $m$ of a module $_RM$ in $R$. We will simply write $l(m)$ or $r(m)$ if the underlying ring is obvious. We recall that an $R$-module is faithful if $l_R(m)=0$ for all $m\in M$.
	
	So far we imposed conditions to explore the structure of the JReject of  a class of modules in a module. We are now intrigued to observe the structure of the JReject of a particular module in the underlying ring.

	\begin{proposition}\label{jerect is annihilator}
		For each $R$-module $M$, the submodule $M/Rad(M)$ is faithful iff $M/Rad(M)$ cogenerates $R$.
	\end{proposition}
	\begin{proof}
		Recalling the fact that every $R$-module $M$ is isomorphic to $Hom_R(R,M) $ \cite[Proposition 4.5]{Anderson}, we may write
		\begin{align*}
		JRej_R(M) & =  \bigcap \{ r \in R ~ \vert ~ rf \in Rad(M), ~ \forall ~ f \in Hom_R(R,M) \}
		\\ & = \bigcap \{ r \in R ~ \vert rm \in Rad(M), ~ \forall ~ m \in M \}
		\\ & = \underset{m \in M} \bigcap {l_R (m+Rad(M))}
		\\ &={l_R(M/Rad(M))}.
		\end{align*}
		Thus assuming Corollary \ref{corgen}, we obtain $M/Rad(M)$ is faithful iff $M/Rad(M)$ cogenerates $R$.
	\end{proof}
	In the light of Proposition \ref{jerect is annihilator}, for a class $\mathscr U$ of $R$-modules, we define the JReject of $\mathscr U$ in $R$ as the annihilator of the subclass $\mathscr U^\prime$, where $\mathscr U^\prime$ remains as previously mentioned. This further leads us to the following corollary.
	
	\begin{corollary}
		For each class $\mathscr U$ of $R$-modules, $JRej_R(\mathscr U)$ is a two-sided ideal of $R$.
	\end{corollary}
	
	We conclude this section by defining the JReject of a ring $R$ in an $R$-module $M$, $JRej_M(R)$ and record a lemma establishing a comparison between radical of $M$ and JReject of $R$ in $M$. This definition will serve as the base of the next section.
		$$JRej_M(R)=\{ m \in M ~ \vert ~ mM^* \subseteq J(R)\}.$$
		
			\begin{lemma}
				For a module $_RM$, Rad(M) $\subseteq JRej_M(R).$
			\end{lemma}
			\begin{proof}
				It follows from Proposition 9.14 of \cite{Anderson} that for any $\theta \in M^*$, $(Rad(M)) \theta  \subseteq J(R).$ Thus {\em Rad(M)} $\subseteq JRej_M(R).$
			\end{proof}
			The converse, however is not true in general. For instance, consider $R= \mathbb{Z}$, $M=\mathbb{Z}_2$. Then $Rad(M)=0$ while $JRej_M(R)= \mathbb{Z}_2.$ However, when $M=R$, $Rad(R)=JRej(R)=J(R)$.

	\section{J-torsionless modules}
	In this section, we introduce the class of {\em J-torsionless} modules which serves as a subclass of the class of torsionless modules. The torsionless modules over a ring $R$, introduced by H. Bass (see \cite{Lam2}), are subdirect products of copies of $R$. This indicates that many of the properties exhibited by a ring are transferred to the torsionless modules over the ring. For instance, a torsionless module becomes Z-regular or semiprime if the underlying ring is von Neumann regular or semiprime, repectively. In ring and module theory, it is customary to analyze the existing module notions in terms of the factor modules $M/Rad(M)$ for any module $M$, or the factor ring $R/J(R)$. With this view in mind, we introduce the notion of {\em J-torsionless} modules as follows.
	
	\begin{definition}
		An $R$-module $M$ is {\em J-torsionless} if $JRej_M(R)=0.$
	\end{definition}
	\begin{remarks} \label{rejrem2}
		\begin{enumerate} \item  An $R$-module is {\em J- torsionless} $\iff mM^* \nsubseteq J(R)$ for all $m \neq 0.$
			\item It is easy to check that the module $_RR$ is J-torsionless if and only if $R$ is semiprimitive. Thus a ring that is semi-simple as a module over itself is also J-torsionless. But the converse is not true as the ring of integers $\mathbb{Z}$ is semiprimitive and hence  $_\mathbb{Z}\mathbb{Z}$ is a J-torsionless module while it is not semi-simple. 
			\item Since $Rej_M(R) \subseteq JRej_M(R)$,  if $M$ is J-torsionless then it is  torsionless.
		\end{enumerate}
	\end{remarks}
	\begin{proposition}\label{base}
		Following are some basic properties of a J-torsionless module.
		\begin{enumerate}
			\item[\rm (1)] Submodules of a J-torsionless $R$-module are again J-torsionless.
			\item[\rm (2)] Direct sums of J-torsionless $R$-modules are J-torsionless.
			\item[\rm (3)] Direct summands of a J-torsionless $R$-module are J-torsionless.
			\item[\rm(4)] Direct products of J-torsionless $R$-modules are J-torsionless.
		\end{enumerate}
	\end{proposition}
	
	\begin{proof} 
		(1) Let $_RN$ be a submodule of a J-torsionless module $_RM$. Take an arbitrary non-zero element $n_1$ of $N$. Given that $M$ is J-torsionless, there exists $q \in M^*$ such that $n_1 q\notin J(R)$. Now we define $\tilde{q} \in N^*$ such that $n \tilde{q}=n q$, for all $n \in$ $N$. This yields that $n_1 \tilde{q}   =n_1 q  \notin J(R)$ and hence, $N$ is a J-torsionless module.
		\medskip	\\
		(2) Let $M= \bigoplus \limits_{i \in I} M_i$ be an $R$-module such that each $M_i$ is J-torsionless, where $I$ is an indexed set. Consider $(m_i)_{i\in I}$ to be a non-zero element in $M$, where $m_i \in M_i$, for each $i \in I$. Then $m_j \neq 0$ for some $j \in I$. Since $M_j$ is J-torsionless, there exists a $q_j \in M_j^*$ such that $m_j q_j \notin J(R)$. Similar to the case of (1), we choose $q \in M^*$ such that $ (m^\prime_i) q=  m^\prime_j q_j$, for all $(m^\prime_i) \in$ $M$. Thus $(m_i) q=m_j q_j \notin J(R)$, implying that $M$ is J-torsionless.
		\medskip
		\\
		(3) Follows immediately from (1).
		\medskip
		\\
		(4) Follows a similar approach to the proof outlined in (2).
	\end{proof}
	It is evident from Corollary \ref{corgen} that an $R$-module $M$ is J-torsionless if and only if it is cogenerated by $R/J(R)$. We now establish that such an $M$ is indeed a submodule of a direct product of copies of $R/J(R)$.
	
	\begin{theorem}
		An $R$-module $M$ is J-torsionless if and only if $M$ is a submodule of a direct product of copies of $R / J(R).$
	\end{theorem}
	\begin{proof}
		Let $\Phi : M \rightarrow{} \prod \limits_I R/J(R)$, where $I$ is an indexed set such that $m \Phi =(m \theta + J(R))_{\theta \in M^*}$. Notice that $\Phi$ is one-one as $0=m \Phi$ implies $m \theta  \in J(R)$ for all $\theta \in M^*$ yielding that $m \in JRej(M)$. Since $M$ is J-torsionless, $m=0$. Conversely, consider $M \xhookrightarrow{} \prod \limits_I R/J(R)$ and $0 \neq m = (x_i + J(R))_{i \in I}\in M$. Then there exists some $j \in I$ such that $x_j \notin J(R).$ Take the canonical projection $\pi_j : M \rightarrow R$ where $m \mapsto x_j$. Then $m \pi_j \notin J(R)$. This establishes that $M$ is J-torsionless.
	\end{proof} 
	
	Recall that a module $M$ is {\em Z-regular} if given any $m \in M$, there exists $q \in M^*$ satisfying $m=(m q) m$, \cite{zel}. A module $_RM$ is {\em anti-regular} when given $0 \neq m \in M$, there exists $0 \neq q \in M^*$ such that $q m q = q$, \cite{devi}.
	
	Since for every non-zero element $m $ in an anti-regular module $_R M$, there exists a non-zero element $q \in M^*$ such that $m q \in I(R)$ and hence, $m q \notin J(R)$, antiregular modules are J-torsionless.  Thus we have the following strict inclusions.
	\[\begin{Bmatrix}
	\text{Regular}
	\end{Bmatrix} \subsetneq \begin{Bmatrix}
	\text{Anti-regular}
	\end{Bmatrix} \subsetneq \begin{Bmatrix}
	\text{J-torsionless}
	\end{Bmatrix} \subsetneq \begin{Bmatrix}
	\text{Torsionless}
	\end{Bmatrix}\]
	
	In what follows, we provide some examples to prove that these inclusions are indeed strict.
	\begin{example}
		Consider the ring $R =\lbrace (a_i)_{i \in \mathbb{N}} \ \vert \ \text{each} ~ a_i \in \mathbb{Q},$ and the sequence eventually assumes a constant integral value$\rbrace$. Then the module $_RR$ is anti-regular but is not regular as the constant sequence with $a_i =2$ for all $i \in \mathbb{N}$ is not regular in $_RR$.
	\end{example}
	
	\begin{example}
		$ _\mathbb{Z}\mathbb{Z}$ is J-torsionless but not anti-regular since $3 \in \mathbb{Z}$ is not anti-regular in $_\mathbb{Z} \mathbb{Z}$.
	\end{example}
	
	\begin{example}
		Every ring is torsionless as a module over itself. Hence in particular, a local ring is torsionless as a module over itself but is not J-torsionless, since $ JRej_R(R )\neq 0$.
	\end{example}

	Evidently, from Remark \ref{rejrem1}, a torsionless module $M$ over a semiprimitive ring is J-torsionless. However, a module over a semiprimitive ring may not be J-torsionless if it is not torsionless over the ring. For instance, $\mathbb{Q}$ and $\mathbb{Z}_2$ are not J-torsionless as modules over $\mathbb{Z}.$ 
	Below we record some conditions under which torsionless implies J-torsionless.

		Let us first recall the definition of a {\em W-regular} module. A projective module is W-regular if every cyclic submodule is a direct summand, \cite{Ware}.
	
	\begin{lemma}\label{w regular is jtorsionless}
		If $_RM$ is W-regular, then it is J-torsionless.
	\end{lemma}
	
	\begin{proof}
		Let $0 \neq m \in M$. Then $Rm \leq^{\oplus}M$. Also, $Rm$ has a maximal submodule $N$ as $Rm$ is finitely generated. This implies that $Rm/N$ is simple and hence isomorphic to $R/\mu$ for some maximal left ideal $\mu $ of $R$.  Now, consider the diagram
		
		\[\begin{tikzcd}
		M \arrow[r, "\pi"] & Rm \arrow[r, "\eta_1"] & Rm/N \stackrel{\phi}{\cong} R/\mu & \arrow[l,"\eta"] R
		\end{tikzcd},\]
		where $\pi$ is the projection map and $\eta_1$ and $\eta $ are the natural homomorphisms. Since $Rm$ is projective being a direct summand of projective, there exists $q: Rm \rightarrow R$ such that the following diagram commutes.
		
		\[
		\begin{tikzcd}
		Rm \arrow[r, "\eta_1\phi"] \arrow[rd, dotted, "q"'] & R/\mu \\
		& R \arrow[u, "\eta"']
		\end{tikzcd}
		\] We claim that $mq \notin \mu$. For if $mq \in \mu,$ then $m(q\eta)=0$ in $R/\mu$. This indeed implies $m( \eta_1 \phi)=0$ in $R/\mu$, which is a contradiction as $m\eta_1=m+N \neq 0$. Now, $\pi q \in M^*$ with $m(\pi q)=mq \notin \mu$. Thus $_RM$ is J-torsionless.
	\end{proof}
	
	A well known result due to Brauer asserts: If $K$ is a minimal left ideal of a ring $R$, then either $K^2=0$ or $K=Re$ for some non-zero $e \in I(R)$, \cite{LamNCR}. The next proposition is an immediate consequence of this result.
	
	\begin{proposition}\label{semiprime}
		If $K$ is a minimal left ideal of a semiprime ring $R$, then $K=Re$ for some $e \in I(R).$
	\end{proposition}

	\begin{theorem}\label{projective simple}
		If a simple module over a semiprime ring is torsionless, then it is J-torsionless.
	\end{theorem}	
		
		\begin{proof}
			Assume that $R$ is a semiprime ring and $T$ is a simple, torsionless $R$- module. Then for a non-zero element $m\in T$, $mq \neq 0$ for some $q \in T^*$. Since $T=Rm$ is simple, it follows that $T$ is isomorphic to the minimal left ideal $Tq$ of $R$. As $R$ is semiprime ring, by Proposition \ref{semiprime},  $Tq=Re$ for some $e\in I(R),$ and the map $q:T \rightarrow Re$ is an isomorphism, yielding $T$ is a projective module. Thus $T$ is W-regular and hence it follows from Lemma \ref{w regular is jtorsionless} that $T$ is a J-torsionless module.
		\end{proof}
		
		This result can be extended to the following Proposition.
		
		\begin{proposition}
			If a semi-simple module over a semiprime ring is torsionless, then it is J-torsionless. 
		\end{proposition}
		\begin{proof}
If $_RM$ is a semi-simple torsionless module, then each of its summands is simple torsionless and hence by Proposition \ref{projective simple}, a J-torsionless module. Thus by (2) of Proposition \ref{base}, $_RM$ is J-torsionless.
		\end{proof}

\begin{proposition}\label{finitely generated}
	Let $M$ be a finitely generated module over a semiprime ring $R$. If every simple factor module of $M$ is torsionless then $M$ is J-torsionless.
\end{proposition}
\begin{proof}
	Let $N$ denote a maximal submodule of $_RM$. By the proof of Theorem \ref{projective simple}, the factor module $M/N$ is projective. This implies that the epimorphism $M \rightarrow M/N$ is split and so, $N$ is a direct summand of $M$. Therefore, it follows that $M$ is a semi-simple, projective module and hence, is J-torsionless.
\end{proof}

It may be remarked that semiprimeness cannot be omitted from the hypothesis of Proposition \ref{finitely generated}, for if $R=\mathbb{Z}_4=M,$ then $J(R)= \langle \bar{2} \rangle$. Therefore $\mathbb{Z}_4$ is not J-torsionless as a module over itself, while every simple factor module of $\mathbb{Z}_4$ is J-torsionless.

\begin{proposition}\label{semisimple ring}
	Let $R$ be a semiprime ring. If all simple $R$-modules are torsionless (equivalently, J-torsionless) then $R$ is a semi-simple ring.
\end{proposition}

\begin{proof}
	Let $\mu $ be a maximal left ideal of $R$. Since $R/ \mu$ is a torsionless $R$-module, $\mu$ is a direct summand of $R$. Thus, every arbitrarily chosen maximal left ideal of $R$ is a direct summand and therefore, $R$ is a semi-simple ring.
\end{proof}

Using Proposition \ref{semisimple ring} and the Wedderburn-Artin theorem (see \cite{LamNCR}) for semi-simple rings, we can deduce:

\begin{corollary} If every simple $R$-module is torsionless, then the following are true for $R$
	\begin{enumerate}
		\item [\rm (i)] If $R$ is (von Neumann) regular, then it is semi-simple.
		\item [\rm (ii)] If $R$ is reduced, then it is direct product of finitely many division rings.
		\item[\rm(iii)] If $R$ is a domain, then it is a division ring.
	\end{enumerate}
\end{corollary}

As observed through examples that regular modules are J-torsionless and hence torsionless but not vice versa, we impose a condition on the latter two classes to create an equivalence among all these three classes of modules. Further, we extend the result : ``every regular module is J-torsionless" in Theorem \ref{fully idempotent} and generate more example of J-torsionless modules.

	\begin{theorem}\label{regular in j-torsionless}
		Given a module $_RM$, following are equivalent:
		\begin{enumerate}
			\item [\rm (1)]$M$ is regular.
			\item [\rm (2)]$M$ is J-torsionless and given any $m \in M$, $mM^*=eR$ for some $e \in I(R)$.
			\item [\rm (3)]$M$ is torsionless and given any $m \in M$, $mM^*=eR$ for some $e \in I(R)$.
		\end{enumerate}
		
	\end{theorem}
	\begin{proof}
		$(1) \implies (2)$ Let $0 \neq m \in M$ such that $m=(m\theta)m$ for some $\theta \in M^*$. Then $m \theta$ is idempotent, hence $m \theta \notin J(R),$ obtaining that $M$ is J-torsionless. Now consider $e=m \theta \in mM^* $. Thus $eR \subseteq mM^*.$ For the converse, take $q \in M^*$. Then
		\[m q =[(m \theta ) m]q =(m \theta)(m q) =e(m q) \in eR. \]
		Thus $mM^* \subseteq eR.$ This proves (2).
		\\ $(2) \implies (3)$ and $(3) \implies (1)$ are consequences of (3) of Remark \ref{rejrem2} and Theorem 2.1 of \cite{chen}, respectively.
	\end{proof}

 Michler and Villamayor in \cite{v ring} called a ring $R$ a left {\em V-ring} if every simple $R$-module is injective. This concept has been extended to modules. An $R$-module $M$ is called {\em V-module} (or a {\em co-simple module}) if every simple $R$-module is $M$-injective in the sense of \cite{azumaya}. An $R$-module $M$ is p-injective if for each element $a \in R$, every $R$-homomorphism from the left ideal $Ra$ to $M$ can be extended to $R$. A ring $R$ is a left {\em SPI}-ring if every simple $R$-module is p-injective. Left V-rings as well as regular rings are clearly SPI.
 
  A ring $R$ is a left {\em fully idempotent} ring if for each element $a \in R$ we have $(Ra)^2=Ra$. It is easily seen that left SPI-rings are left fully idempotent. A module $_RM$ is fully idempotent if $R_m=(RmM^*)m$, for each $m \in M$. This definition extends the fully idempotent concept to modules. Fully idempotent rings and modules were introduced by Ramamurthi (see \cite{ramamurthi} and \cite{ramamurthi2}) who used the term {\em weakly regular} instead of fully idempotent.

\begin{remarks}
	\begin{enumerate}
		\item[(1)] Fully idempotent rings are semiprimitive \cite[Proposition 14(1)]{ramamurthi} and hence they are J-torsionless as modules over themselves, [(2) of Remark \ref{rejrem2}]. 
		\item[(2)] Every simple $\mathbb{Z}$-module is a V-module. Thus V-modules need not be J-torsionless.
	\end{enumerate}
\end{remarks}
 As noted in Theorem \ref{regular in j-torsionless}, regular modules are J-torsionless. Since regular modules are trivially fully idempotent, the following theorem extends the aforementioned result.
	\begin{theorem}\label{fully idempotent}
		If a module $M$ is fully idempotent, then it is J-torsionless.
	\end{theorem}
	
	\begin{proof}
		Let $m \neq 0$ and assume that $mM^* \subseteq J(R)$. This implies that $RmM^* \subseteq J(R)$. As $M$ is fully idempotent, we have $Rm \subseteq J(R)m$, yielding that there exists an element $r \in J(R)$ such that $m=rm$. Thus, $(1-r) \in {l_R(m)}$. Now, as $m\neq 0$, it ensures that $l_R(m) \subset R$. Therefore, there exists a maximal left ideal $\mu$ in $R$ such that $l_R(m) \subseteq \mu$ and hence, $(1-r) \in \mu$. This leads us to a contradiction as $r \in \mu$, thus implying that $1 \in \mu$. Therefore, there exists some $q \in M^*$ with $mq \notin J(R)$.
	\end{proof}

		We bring this paper to an end by characterizing the ideals (respectively, maximal ideals) of a ring over which every cyclic (respectively, semi-simple) module is J-torsionless. Finally, we provide a characterization of a self-injective ring over which every module is J-torsionless.
		
	\begin{remarks}
		\begin{enumerate}	\label{rejrem3}
			\item[(1)]	Let $A$ and $B$ be subsets of a ring $R$. Then $(A:B)_r$ is the set of elements $t \in R$ satisfying $tB \subsetneq A$. If $A$ and $B$ are subgroups of $(R, +)$ then so is $(A:B)_r$. There is, of course a left-sided counter part of this notation.
			\item[ (2)] Lam \cite{Lam2} has noted that for a left ideal $I$ of a ring $R$, the left $R$-modules $R/I$ is torsionless if and only if $I=(0:r_R(I))_l$.
			\item[(3)]  The result in Remark (2) above is extended in Proposition \ref{procy} by proving that for a left ideal $I$ of a ring $R$, the $R$-module $R/I$ is J-torsionless if and only if $I=(J(R):r_R(I))_l$.
			\item[(4)]	A notable observation is that any non trivial module homomorphism $q$ from a cyclic $R$-module $R/ I$ to $R$ takes each element $x+ I$ to the element $x[(1+I)q]$ of $R$, where $(1+I)q \neq0$. Thus $I \subseteq Ker(q)$. We claim that $(1+I)q \in$ $r_R(I)$. For consider $t \in I$ and $a=(1+I)q$. Then $(t+ I)q = ((1+t)-1+I)q=(1+t)a-a=ta=0$. Since it is true for arbitrary $t \in I$, thus $a \in$ $r_R(I)$.
			\item[(5)] We refer to \cite{Hera} to recall the necessary concepts used in this remark. Suppose that the module $_RM$ is an {\em NS-module}. Let $m \in NilRej(M)$ and $q \in M^*$ so that $mq \in Nil(R)$. As {\em NilRej(M)} is an $R$-submodule of $M$ for each element $t \in R$ we have $t(mq)=(tm)q$ is also nilpotent. Hence $1-t(mq) \in U(R)$ for each $t \in R$ yielding $mq \in J(R)$. This proves that $NilRej(M)\leq JRej(M)$. It follows therefore, that every J-torsionless NS-module is niltorsionless.
		\end{enumerate}
	\end{remarks}
	
		\begin{theorem} \label{procy}
			For a ring $R$, following are equivalent:
			\begin{enumerate}
				\item[\rm (1)] Every cyclic left $R$-module is J-torsionless.
				\item[\rm (2)] $I=(J(R):r_R(I))_l, $ for any left ideal $I$ in $R$. 
			\end{enumerate}
		\end{theorem}
		\begin{proof}$(1) \implies (2)$ Let $I$ be any left ideal $I$ of $R$. It is clear that $I \subseteq (J(R):r_R(I))_l$. Assume $x \in (J(R):r_R(I))_l$. If possible, take $x \notin I$. As $R/I$ is J-torsionless, there exists $q \in (R/I)^*$ such that $(x+I)q \notin J(R)$. 
			By (4) of Remark \ref{rejrem3}, there exists $0 \neq a \in$ $r_R(I)$ such that $xa \notin J(R)$, which yields that $x \notin (J(R):r_R(I))_l$, which is a contradiction. Thus $x \in I$ and consequently, $I=(J(R):r_R(I))_l.$
			
			$(2) \implies (1)$ Since any cyclic module $_RM$ is isomorphic to the $R$-module $R/I$ for some left ideal $I$ of $R$, it is enough for us to show that given any left ideal $I$ of $R$, $R/I$ is J-torsionless under the assumption in (2). Consider an arbitrary left ideal $I$ of $R$ and choose a non-zero element $x+I$ of $R/I$. Then $x \notin I= (J(R):r_R(I))_l$ implying that there exists $a \in r_R(I)$ such that $xa \notin J(R)$. Consider $q \in (R/I)^*$ such that $(1+I)q=a.$ Assuming (4) of Remark \ref{rejrem3}, we have $(x+I)q=xa \notin J(R)$. This establishes that $R/I$ is J-torsionless. 
		\end{proof}
			We would like to recall here the definition of a {\em left annihilator ring} or simply an {\em LA-ring}. A ring $R$ is said to be an {\em LA- ring} if each left ideal of $R$ is a left annihilator in $R$.
				\begin{corollary}\label{charac}
					If every cyclic module over a ring $R$ is J-torsionless, then $R$ is a semiprimitive LA-ring. 
				\end{corollary}	
				\begin{proof}
				Clearly, $_RR$ is J-torsionless and therefore considering Remark \ref{rejrem2}, $R$ is a semiprimitive ring. Furthermore, since all cyclic modules over $R$ are J-torsionless, as a consequence of Proposition \ref{procy} we have, any left ideal $I$ in $R$ satisfies $I=(0:r_R(I))=l_R(r_R(I))$. Hence, $R$ is an LA-ring.
				\end{proof}
				
		In light of the above Theorem \ref{procy}, we document the following theorem where we characterize the maximal left ideals of a ring over which every simple module is J-torsionless.
		\begin{theorem} \label{T1}
			For a ring $R$, following are equivalent:
			\begin{enumerate}
				\item[\rm (1)] Every semi-simple $R$-module is J-torsionless.
				\item [\rm (2)] Every simple $R$-module is J-torsionless.
				\item [\rm (3)] For each maximal left ideal $\mu$ in $R$, $\mu =(J(R):r_R(\mu))_l.$
			\end{enumerate}
		\end{theorem}
		
		\begin{proof}
			The implication $(1) \implies (2)$ is trivial.
			\medskip
			\\
			$(2) \implies (1)$ follows from (1) of Proposition \ref{base}.
			\medskip
			\\
			$(2) \iff (3)$ similar to Theorem \ref{procy}.
		\end{proof}

		\begin{theorem}
			Over a self-injective ring, following conditions are equivalent:
			
			\begin{enumerate}
				\item[\rm (1)] Every module $_RM$ is J-torsionless.
				\item [\rm (2)] Every finitely generated module $_RM$ is J-torsionless.
				\item [\rm (3)]Every cyclic module $_RM$ is J-torsionless.
				\item [\rm (4)]$R$ is a semiprimitive LA ring.
			\end{enumerate}
		\end{theorem}
		
		\begin{proof}
			The implications $(1) \implies (2) \implies (3)$ are direct.
			\medskip
			\\ 
			$(3) \implies (4)$  Let $I$ be a non-zero left ideal of $R$. Note that $_RR$ is a J-torsionless module. Therefore, by Remark \ref{rejrem2}, $R$ is a semiprimitive ring, i.e., $J(R)=0$. Again, by assumption, $R/I$ is J-torsionless. Thus using Proposition \ref{procy} we have $I=(J(R):r_R(I))_l=(0:r_R(I))_l=l_R(r_R(I))$. This proves that any left ideal of $R$ is a left annihilator in $R$.
			\medskip	\\ $(4) \implies (3)$
			Assume $R/I$ is any cyclic $R$-module for some left ideal $I$ of $R$. Since $R$ is a semiprimitive LA ring, $r_R(I)\neq 0$ and also $I=l_R(r_R(I))=(0:r_R(I))_l=(J(R):r_R(I))_l$. Since this is true for any arbitrary left ideal of $R$, using Proposition \ref{procy}, we may conclude that $R/I$ is J-torsionless.
			\medskip
			\\
			$(3) \implies (1)$ Let $M$ be a left $R$-module and $0 \neq m \in M$. Then $Rm$ is a cyclic $R$-module and hence by assumption is J-torsionless. Also $m \in Rm$ implies that there exists a non-zero $q \in (Rm)^*$ such that $mq \notin J(R)$. Since $R$ is self-injective, there exists $\Tilde{q}\in M^*$ such that $i \circ \Tilde{q}  =q$ for the inclusion map $i:Rm \hookrightarrow M$. Thus $m\Tilde{q} = m (i \circ \Tilde{q} ) =mq \notin J(R)$, yielding that $_RM$ is J-torsionless.
		\end{proof}
\section{Conflict of Interest} 
On behalf of all authors, the corresponding author states that there is no conflict of interest regarding the publication of this manuscript. The research work presented in this manuscript was conducted independently, and there are no financial or personal relationships that could influence the results or interpretations of the study. Additionally, all authors have reviewed and approved the final manuscript and agree to its submission to the journal.
	

\begin{thebibliography}{33}
		\bibitem{Anderson}{F.W.Anderson and K.R.Fuller, {\em Rings and Categories of Modules}, Springer-Verlag, 1974.}
		\medskip
		\bibitem{zel} J.Zelmanowitz, {\em Regular modules}, Transactions of the American Mathematical Society, \textbf{163}(1972), 341-355.
		
		\medskip
		\bibitem{chen} H.Chen and W.K.Nicholson, {\em On regular modules}, Bulletin of Australian Mathematical Society, \textbf{84}(2011), 280-287.
		\medskip
		\bibitem{devi} I.Choudhuri, M.I.Devi and M.B.Rege, {\em Anti-regularity in modules}, National Academy Science Letters, \textbf{13}(6) (1990), 279-281.
		\medskip
		\bibitem{MJ} M. Jaegermann and J. Krempa, {\em Rings in which ideals are annihilators}, Fundamenta Mathematicae, \textbf{76}(2) (1972), 95-107
		\medskip
			\bibitem{Lam2} T. Y. Lam, {\em Lecturs on Modules and Rings - Graduate Texts in Mathematics}, {Springer-Verlag}, \textbf{189} (1999).
			\medskip
			
		\bibitem{Hera}
		Kh.H.Singh, A.M.Buhphang and M.B.Rege, {\em Armendariz modules and niltorsionless modules}, Mathematical Reports, \textbf{25}(75) (2023), 2, 319-329.
		\medskip
		
		\bibitem{LamNCR} T. Y. Lam, {\em A First Course in Noncommutative Rings} (Second edition), {Springer-Verlag}, (2001).
	\medskip
	
	\bibitem{Ware} R. Ware, {\em Endomorphism ring of projective modules}, {Transactions of the American Mathematical Society}, \text{155}(1) (1971), 233-256.
	\medskip
	
	\bibitem{v ring} G. O. Michler, O. E. Villamayor, {\em On rings whose simple modules are injective}, {Journal of Algebra}, \textbf{25}(1) (1973), 185-201.
\medskip

\bibitem{azumaya} G. Azumaya, F. Mbuntun, K. Varadarajan, {\em On M-projective and M-injective modules}, {Pacific Journal of Mathematics}, \textbf{59}(1) (1957), 9-16.
\medskip

\bibitem{ramamurthi}  V. S. Ramamurthi, {\em Weakly regular rings}, Canadian Mathematical Bulletin, \textbf{16}(3) (1973), 317-321.
\medskip

\bibitem{ramamurthi2}  V. S. Ramamurthi, {\em A note on regular modules}, Bulletin of Australian Mathematical Society, \textbf{11} (1974), 359-364.

	\end{thebibliography}
\end{document}